\documentclass[12pt]{amsart}

\usepackage{fullpage}
\usepackage{amssymb}
\usepackage{color}

\theoremstyle{plain}
\newtheorem{theorem}{Theorem}[section]

\newtheorem{corollary}[theorem]{Corollary}

\newtheorem{lemma}[theorem]{Lemma}

\newtheorem{proposition}[theorem]{Proposition}

\newtheorem{construction}[theorem]{Construction}

\theoremstyle{definition}

\numberwithin{equation}{section}

\def\End{\mathrm{End}}
\def\Aut{\mathrm{Aut}}
\def\Mlt{\mathrm{Mlt}}
\def\Inn{\mathrm{Inn}}
\def\End{\mathrm{End}}
\def\Asc{\mathrm{Asc}}
\def\F{\mathbb F}
\def\GL{\mathrm{GL}}

\def\im{\mathrm{Im}}

\def\tr{\mathrm{tr}}
\def\det{\mathrm{det}}
\def\chr#1{\mathrm{char}(#1)}

\begin{document}

\title{Automorphic loops arising from module endomorphisms}

\author{Alexandr Grishkov}

\email[Grishkov]{grishkov@ime.usp.br}

\address[Grishkov]{Institute of Mathematics and Statistics, University of S\~{a}o Paulo, 05508-090 S\~{a}o Paulo, SP, Brazil}

\author{Marina Rasskazova}

\email[Rasskazova]{marinarasskazova1@gmail.com}

\address[Rasskazova]{Omsk Service Institute, 644099 Omsk, Russia}

\author{Petr Vojt\v{e}chovsk\'y}

\email[Vojt\v{e}chovsk\'y]{petr@math.du.edu}

\address[Vojt\v{e}chovsk\'y]{Department of Mathematics, University of Denver, 2280 S Vine St, Denver, Colorado 80208, U.S.A.}

\begin{abstract}A loop is automorphic if all its inner mappings are automorphisms. We construct a large family of automorphic loops as follows. Let $R$ be a commutative ring, $V$ an $R$-module, $E=\End_R(V)$ the ring of $R$-endomorphisms of $V$, and $W$ a subgroup of $(E,+)$ such that $ab=ba$ for every $a$, $b\in W$ and $1+a$ is invertible for every $a\in W$. Then $Q_{R,V}(W)$ defined on $W\times V$ by
\begin{displaymath}
    (a,u)(b,v) = (a+b,u(1+b)+v(1-a))
\end{displaymath}
is an automorphic loop.

A special case occurs when $R=k<K=V$ is a field extension and $W$ is a $k$-subspace of $K$ such that $k1\cap W = 0$, naturally embedded into $\End_k(K)$ by $a\mapsto M_a$, $bM_a = ba$. In this case we denote the automorphic loop $Q_{R,V}(W)$ by $Q_{k<K}(W)$.

We call the parameters tame if $k$ is a prime field, $W$ generates $K$ as a field over $k$, and $K$ is perfect when $\chr{k}=2$. We describe the automorphism groups of tame automorphic loops $Q_{k<K}(W)$, and we solve the isomorphism problem for tame automorphic loops $Q_{k<K}(W)$. A special case solves a problem about automorphic loops of order $p^3$ posed by Jedli\v{c}ka, Kinyon and Vojt\v{e}chovsk\'y.

We conclude the paper with a construction of an infinite $2$-generated abelian-by-cyclic automorphic loop of prime exponent.
\end{abstract}

\thanks{Research partially supported by the Simons Foundation Collaboration Grant 210176 to Petr Vojt\v{e}chovsk\'y.}



\maketitle

\section{Introduction}

A groupoid $Q$ is a \emph{quasigroup} if for all $x\in Q$ the translations $L_x:Q\to Q$, $R_x:Q\to Q$ defined by $yL_x = xy$, $yR_x = yx$ are bijections of $Q$. A quasigroup $Q$ is a \emph{loop} if there is $1\in Q$ such that $1x=x1=x$ for every $x\in Q$.

Let $Q$ be a loop. The \emph{multiplication group} of $Q$ is the permutation group $\Mlt(Q) = \langle L_x,\,R_x:x\in Q\rangle$, and the \emph{inner mapping group} of $Q$ is the subgroup $\Inn(Q) = \{\varphi\in \Mlt(Q):1\varphi=1\}$.

A loop $Q$ is said to be \emph{automorphic} if $\Inn(Q)\le \Aut(Q)$, that is, if every inner mapping of $Q$ is an automorphism of $Q$. Since, by a result of Bruck \cite{Br}, $\Inn(Q)$ is generated by the bijections
\begin{displaymath}
    T_x = R_x L_x^{-1},\quad L_{x,y} = L_x L_y L_{yx}^{-1},\quad R_{x,y} = R_x R_y R_{xy}^{-1},
\end{displaymath}
a loop $Q$ is automorphic if and only $T_x$, $L_{x,y}$, $R_{x,y}$ are homomorphisms of $Q$ for every $x$, $y\in Q$. In fact, by \cite[Theorem 7.1]{JoKiNaVo}, a loop $Q$ is automorphic if and only if every $T_x$ and $R_{x,y}$ are automorphisms of $Q$. The variety of automorphic loops properly contains the variety of groups.

See \cite{Br} or \cite{Pf} for an introduction to loop theory. The first paper on automorphic loops is \cite{BrPa}. It was shown in \cite{BrPa} that automorphic loops are power-associative, that is, every element of an automorphic loop generates an associative subloop. Many structural results on automorphic loops were obtained in \cite{KiKuPhVo}, where an extensive list of references can be found.

\subsection{The general construction}

In this paper we study the following construction.

\begin{construction}\label{Cn:Main}
Let $R$ be a commutative ring, $V$ an $R$-module and $E=\End_R(V)$ the ring of $R$-endomorphisms of $V$. Let $W$ be a subgroup of $(E,+)$ such that
\begin{enumerate}
\item[(A1)] $ab=ba$ for every $a$, $b\in W$, and
\item[(A2)] $1+a$ is invertible for every $a\in W$,
\end{enumerate}
where $1\in E$ is the identity endomorphism on $V$.

Define $Q_{R,V}(W)$ on $W\times V$ by
\begin{equation}\label{Eq:Mult}
    (a,u)(b,v) = (a+b, u(1+b)+v(1-a)).
\end{equation}
\end{construction}

We show in Theorem \ref{Th:A} that $Q_{R,V}(W)$ is always an automorphic loop.

\medskip

Two special cases of this construction appeared in the literature. First, in \cite{JeKiVoNilp}, the authors proved that commutative automorphic loops of odd prime power order are centrally nilpotent, and constructed a family of (noncommutative) automorphic loops of order $p^3$ with trivial center by using the following construction.

\begin{construction}\label{Cn:JeKiVo}
Let $k$ be a field and $M_2(k)$ the vector space of $2\times 2$ matrices over $k$ equipped with the determinant norm. Let $I$ be the identity matrix, and let $A\in M_2(k)$ be such that $kI\oplus kA$ is an anisotropic plane in $M_2(k)$, that is, $\det(aI+bA)\ne 0$ for every $(a,b)\ne (0,0)$. Define $Q_k(A)$ on $k\times(k\times k)$ by $(a,u)(b,v) = (a+b,u(I+bA)+v(I-aA))$.
\end{construction}

We will show in Section \ref{Sc:p3} that the loops $Q_k(A)$ are a special case of the construction $Q_{R,V}(W)$ and hence automorphic. If $k=\F_p$ then $Q_k(A)$ has order $p^3$, exponent $p$ and trivial center, by \cite[Proposition 5.6]{JeKiVoNilp}.

\medskip

Second, in \cite{Na}, Nagy used a construction of automorphic loops based on Lie rings (cf. \cite{KiGrNa} and \cite{KiKuPhVo}) and arrived at the following.

\begin{construction}\label{Cn:Nagy}
Let $V$, $W$ be vector spaces over $\F_2$, and let $\beta:W\to\End(V)$ be a linear map such that $a\beta b\beta = b\beta a\beta$ for every $a$, $b\in W$, and $1+a\beta$ is invertible for every $a\in W$. Define a loop $(W\times V,*)$ by $(a,u)*(b,v) = (a+b,u(1+b\beta)+v(1+a\beta))$.
\end{construction}

When $\beta$ is injective, Construction \ref{Cn:Nagy} is a special case of our Construction \ref{Cn:Main}, and when $\beta$ is not injective, it is a slight variation. By \cite[Proposition 3.2]{Na}, $(W\times V,*)$ is an automorphic loop of exponent $2$ and, moreover, if $\beta$ is injective and at least one $a\beta$ is invertible then $(W\times V,*)$ has trivial center.

\subsection{The field extension construction}

Most of this paper is devoted to the following special case of Construction \ref{Cn:Main}.

\begin{construction}\label{Cn:Extension}
Let $R=k<K=V$ be a field extension, and let $W$ be a $k$-subspace of $V$ such that $k1\cap W = 0$. Embed $W$ into $\End_k(K)$ via $a\mapsto M_a$, $bM_a=ba$. Denote by $Q_{k<K}(W)$ the loop $Q_{R,V}(W)$ of Construction \ref{Cn:Main}.
\end{construction}

Assuming the situation of Construction \ref{Cn:Extension}, the condition (A1) of Construction \ref{Cn:Main} is obviously satisfied because the multiplication in $K$ is commutative and associative. Moreover, $k1\cap W=0$ is equivalent to $1+a\ne 0$ for all $a\in W$, which is equivalent to (A2). Construction \ref{Cn:Main} therefore applies and $Q_{k<K}(W)$ is an automorphic loop.

For the purposes of this paper, we call the parameters $k$, $K$, $W$ of Construction \ref{Cn:Extension} \emph{tame} if $k$ is a prime field, $W$ generates $K$ as a field over $k$, and $K$ is perfect when $\chr{k}=2$.

In Corollary \ref{Cr:Iso} we solve the isomorphism problem for tame automorphic loops $Q_{k<K}(W)$, given a fixed extension $k<K$, and in Theorem \ref{Th:Aut} we describe the automorphism groups of tame automorphic loops $Q_{k<K}(W)$. In particular, we solve the isomorphism problem when $k$ is a finite prime field and $K$ is a quadratic extension of $k$. This answers a problem about automorphic loops of order $p^3$ posed in \cite{JeKiVoNilp}, and it disproves \cite[Conjecture 6.5]{JeKiVoNilp}.

Finally, in Section \ref{Sc:Example} we use the construction $Q_{k<K}(W)$ to obtain an infinite $2$-generated abelian-by-cyclic automorphic loop of prime exponent.

\section{Automorphic loops from module endomorphisms}

Throughout this section, assume that $R$ is a commutative ring, $V$ an $R$-module, $W$ a subgroup of $E=(\End_R(V),+)$ satisfying (A1) and (A2), and $Q_{R,V}(W)$ is defined on $W\times V$ by \eqref{Eq:Mult} as in Construction \ref{Cn:Main}.

It is easy to see that $(0,0)=(0_E,0_V)$ is the identity element of $Q_{R,V}(W)$, and that $(a,u)\in Q_{R,V}(W)$ has the two-sided inverse $(-a,-u)$.

Using the notation
\begin{displaymath}
    I_a = 1+a\text{ and }J_a = 1-a,
\end{displaymath}
we can rewrite the multiplication formula \eqref{Eq:Mult} as
\begin{displaymath}
    (a,u)(b,v) = (a+b, uI_b+vJ_a).
\end{displaymath}
A straightforward calculation then shows that the left and right translations $L_{(a,u)}$, $R_{(a,u)}$ in $Q_{R,V}(W)$ are invertible, with their inverses given by
\begin{align}
        (a,u)\backslash (b,v) = (b,v)L_{(a,u)}^{-1} &= (b-a, (v-uI_{b-a})J_a^{-1}),\label{Eq:LeftDivision}\\
        (b,v)/(a,u) = (b,v)R_{(a,u)}^{-1} &= (b-a, (v-uJ_{b-a})I_a^{-1}),\label{Eq:RightDivision}
\end{align}
respectively. Hence $Q_{R,V}(W)$ is a loop.

The multiplication formula \eqref{Eq:Mult} yields $(a,0)(b,0) = (a+b,0)$ and $(0,u)(0,v)=(0,u+v)$, so $W\times 0$ is a subloop of $Q_{R,V}(W)$ isomorphic to the abelian group $(W,+)$ and $0\times V$ is a subloop of $Q_{R,V}(W)$ isomorphic to the abelian group $(V,+)$. Moreover, the mapping $Q_{R,V}(W) = W\times V \to W$ defined by $(a,u)\mapsto a$ is a homomorphism with kernel $0\times V$. Thus $0\times V$ is a normal subloop of $Q_{R,V}(W)$.

We proceed to show that $Q_{R,V}(W)$ is an automorphic loop.

Let $C_E(W) = \{a\in E : ab=ba$ for every $b\in W\}$.

\begin{lemma}\label{Lm:ParamAut}
For $d\in C_E(W)^*$ and $x\in V$ define $f_{(d,x)}:Q_{R,V}(W)\to Q_{R,V}(W)$ by
\begin{displaymath}
    (a,u)f_{(d,x)} = (a,xa+ud).
\end{displaymath}
Then $f_{(d,x)}\in \Aut(Q_{R,V}(W))$.
\end{lemma}
\begin{proof}
We have $((a,u)(b,v))f_{(d,x)} = (a+b,uI_b+vJ_a)f_{(d,x)} = (a+b,x(a+b)+(uI_b+vJ_a)d)$, where the second coordinate is equal to $xa+xb+ud+ubd+vd-vad$. On the other hand, $(a,u)f_{(d,x)}\cdot (b,v)f_{(d,x)} = (a,xa+ud)(b,xb+vd) = (a+b, (xa+ud)I_b + (xb+vd)J_a)$, where the second coordinate is equal to $xa+xab+ud+udb+xb-xba+vd-vda$. Note that $ab=ba$ because $a$, $b\in W$, and $ad=da$, $bd=db$ because $d\in C_E(W)$. The mapping $f_{(d,x)}$ is therefore an endomorphism of $Q_{R,V}(W)$.

Suppose that $(a,u)f_{(d,x)} = (b,v)f_{(d,x)}$. Then $(a,xa+ud) = (b,xb+vd)$ implies $a=b$ and $ud=vd$. Since $d$ is invertible, we have $u=v$, proving that $f_{(d,x)}$ is one-to-one.

Given $(b,v)\in Q_{R,V}(W)$, we have $(a,u)f_{(d,x)} = (b,v)$ if and only if $(a,xa+ud)=(b,v)$. We can therefore take $a=b$ and $u = (v-xa)d^{-1}$ to see that $f_{(d,x)}$ is onto.
\end{proof}

\begin{theorem}\label{Th:A}
The loops $Q_{R,V}(W)$ obtained by Construction \ref{Cn:Main} are automorphic.
\end{theorem}
\begin{proof}
We have already shown that $Q=Q_{R,V}(W)$ is a loop. In view of \cite[Theorem 7.1]{JoKiNaVo}, it suffices to show that for every $(a,u)$, $(b,v)\in Q$ the inner mappings $T_{(a,u)}$, $L_{(a,u),(b,v)}$ are automorphisms of $Q$. Using \eqref{Eq:LeftDivision}, we have
\begin{align*}
    (b,v)T_{(a,u)}&= (b,v)R_{(a,u)}L_{(a,u)}^{-1} = (b+a,vI_a+uJ_b)L_{(a,u)}^{-1}\\
        &= (b,(vI_a+uJ_b-uI_b)J_a^{-1}) = (b,u(J_b-I_b)J_a^{-1}+vI_aJ_a^{-1})\\
        &= (b, -2ubJ_a^{-1}+vI_aJ_a^{-1}) = (b, (-2uJ_a^{-1})b + v(I_aJ_a^{-1})),
\end{align*}
where we have also used $bJ_a^{-1}=J_a^{-1}b$. Thus $T_{(a,u)} = f_{(d,x)}$ with $d=I_aJ_a^{-1}$ and $x=-2uJ_a^{-1}\in V$. Note that $d\in C_E(W)^*$ by (A1), (A2). By Lemma \ref{Lm:ParamAut}, $T_{(a,u)}\in\Aut(Q)$.

Furthermore,
\begin{align*}
    (c,w)L_{(a,u),(b,v)} &= ((b,v)\cdot (a,u)(c,w))L_{(b,v)(a,u)}^{-1}\\
        &= ((b,v)(a+c,uI_c+wJ_a))L_{(b+a,vI_a+uJ_b)}^{-1}\\
        &= (b+a+c,vI_{a+c}+uI_cJ_b+wJ_aJ_b)L_{(b+a,vI_a+uJ_b)}^{-1}\\
        &= (c, (vI_{a+c}+uI_cJ_b+wJ_aJ_b-vI_aI_c-uJ_bI_c)J_{b+a}^{-1})\\
        &= (c, v(I_{a+c}-I_aI_c)J_{b+a}^{-1} + wJ_aJ_bJ_{b+a}^{-1})\\
        &= (c, - vacJ_{b+a}^{-1} + wJ_aJ_bJ_{b+a}^{-1}) = (c, (-vaJ_{b+a}^{-1})c + w(J_aJ_bJ_{b+a}^{-1})).
\end{align*}
Thus $L_{(a,u),(b,v)} = f_{(d,x)}$ with $d=J_aJ_bJ_{b+a}^{-1}\in C_E(W)^*$ and $x=-vaJ_{b+a}^{-1}\in V$. By Lemma \ref{Lm:ParamAut}, $L_{(a,u),(b,v)}\in\Aut(Q)$.
\end{proof}

For a loop $Q$, the \emph{associator subloop} $\Asc(Q)$ is the smallest normal subloop of $Q$ such that $Q/\Asc(Q)$ is a group. Given $x$, $y$, $z\in Q$, the \emph{associator} $[x,y,z]$ is the unique element of $Q$ such that $(xy)z = [x,y,z](x(yz))$, so
\begin{displaymath}
    [x,y,z] = ((xy)z)/(x(yz)) = ((xy)z)R_{x(yz)}^{-1}.
\end{displaymath}
It is easy to see that $\Asc(Q)$ is the smallest normal subloop of $Q$ containing all associators.

\begin{lemma}\label{Lm:Associator}
Let $Q=Q_{R,V}(W)$. Then
\begin{displaymath}
    [(a,u),(b,v),(c,w)] = (0,(ubc-wab)I_{a+b+c}^{-1})
\end{displaymath}
for every $(a,u)$, $(b,v)$, $(c,w)\in Q$. In particular, $\Asc(Q)\le 0\times V$.
\end{lemma}
\begin{proof}
The associator $[(a,u),(b,v),(c,w)]$ is equal to
\begin{align*}
    ((a,u)(b,v)&\cdot(c,w))R_{(a,u)\cdot(b,v)(c,w)}^{-1}\\
    &= (a+b+c,(uI_b+vJ_a)I_c+wJ_{a+b}) R_{(a+b+c,uI_{b+c}+(vI_c+wJ_b)J_a)}^{-1}\\
    &= (0,(uI_bI_c+vJ_aI_c+wJ_{a+b}-uI_{b+c}-vI_cJ_a-wJ_bJ_a)I_{a+b+c}^{-1})\\
    &= (0,(ubc-wab)I_{a+b+c}^{-1}).
\end{align*}
Since $0\times V$ is a normal subloop of $Q$, we are done.
\end{proof}

\begin{corollary}\label{Cr:Group}
Let $Q=Q_{R,V}(W)$.
\begin{enumerate}
\item[(i)] $Q$ is a group if and only if $W^2 = \{ab:a,\,b\in W\}=0$.
\item[(ii)] If $VW^2=V$ then $\Asc(Q)=0\times V$.
\end{enumerate}
\end{corollary}
\begin{proof}
(i) It is clear that $Q$ is a group if and only if $\Asc(Q)=0$. Suppose that $Q$ is a group. Taking $w=0$ and $a=-(b+c)$ in Lemma \ref{Lm:Associator}, we get $[(a,u),(b,v),(c,w)] = (0,ubc)$, so $W^2=0$. Conversely, if $W^2=0$ then the formula of Lemma \ref{Lm:Associator} shows that every associator vanishes.

(ii) As above, with $w=0$ and $a=-(b+c)$ we get $[(a,u),(b,v),(c,w)] = (0,ubc)$. Since $VW^2=V$, we conclude that $0\times V\le\Asc(Q)$. The other inclusion follows from Lemma \ref{Lm:Associator}.
\end{proof}

\section{Automorphic loops from field extensions}

Throughout this section we will assume that $R=k<K=V$ is a field extension, $k$ embeds into $K$ via $\lambda\mapsto \lambda 1$, and $W$ is a $k$-subspace of $K$ such that $k1\cap W=0$, where we identify $a\in W$ with $M_a:K\to K$, $b\mapsto ba$. We write $M_W = \{M_a:a\in W\}$.

We have already pointed out in the introduction that (A1), (A2) are then satisfied, giving rise to the automorphic loop $Q_{k<K}(W)$ of Construction \ref{Cn:Extension}. Note that the multiplication formula \eqref{Eq:Mult} on $W\times K$ makes sense as written even with addition and multiplication from $K$.

\begin{corollary}\label{Cr:Associator}
Let $Q=Q_{k<K}(W)$ with $W\ne 0$. Then $\Asc(Q) = 0\times K$.
\end{corollary}
\begin{proof}
Let $0\ne a\in W$ and note that $M_a$ is a bijection of $V$. Thus $VW^2\supseteq VM_aM_a=V$, and we are done by Corollary \ref{Cr:Group}.
\end{proof}

\subsection{Isomorphisms}

We proceed to investigate isomorphisms between loops $Q_{k<K}(W)$ for a fixed field extension $k<K$. 

Let $W_0$, $W_1$ be two $k$-subspaces of $K$ satisfying $k1\cap W_0 = 0 = k1\cap W_1$. Let
\begin{displaymath}
    S(W_0,W_1) = \{A:\text{$A$ is an additive bijection $K\to K$ and }A^{-1}M_{W_0}A = M_{W_1}\}.
\end{displaymath}
Any $A\in S(W_0,W_1)$ induces the map $\bar A:W_0\to W_1$ defined by
\begin{displaymath}
    A^{-1} M_a A = M_{a\bar A},\quad a\in W_0,
\end{displaymath}
in fact an additive bijection $W_0\to W_1$. Indeed: $\bar A$ is onto $W_1$ by definition; if $a$, $b\in W_0$ are such that $A^{-1}M_aA = A^{-1}M_bA$ then $M_a=M_b$ and $a=1M_a = 1M_b=b$, so $\bar A$ is one-to-one; and $M_{(a+b)\bar A} = A^{-1}M_{a+b}A = A^{-1}(M_a+M_b)A = A^{-1}M_aA + A^{-1}M_bA = M_{a\bar A} + M_{b\bar A}$, so $(a+b)\bar A = a\bar A + b\bar A$.

\begin{proposition}\label{Pr:Correspondence}
For $i\in\{0,1\}$, let $Q_i=Q_{k<K}(W_i)$ with $W_i\ne 0$. Suppose that $K$ is perfect if $\chr{k}=2$. Then there is a one-to-one correspondence between the set $\mathrm{Iso}(Q_0,Q_1)$ of all isomorphisms $Q_0\to Q_1$ and the set $S(W_0,W_1)\times K$. The correspondence is given by
\begin{displaymath}
    \Phi:\mathrm{Iso}(Q_0,Q_1)\to S(W_0,W_1)\times K,\quad f\Phi = (A,c),
\end{displaymath}
where $(A,c)$ are defined by
\begin{displaymath}
    (0,u)f = (0,uA)\text{ and }(a,0)f = (a\bar A, c\cdot a\bar A),
\end{displaymath}
and by the converse map
\begin{displaymath}
    \Psi: S(W_0,W_1)\times K\to \mathrm{Iso}(Q_0,Q_1),\quad (A,c)\Psi = f,
\end{displaymath}
where $f$ is defined by
\begin{equation}\label{Eq:f2}
    (a,u)f = (a\bar A, c\cdot a\bar A + uA).
\end{equation}
\end{proposition}

\begin{proof}
Given $A\in S(W_0,W_1)$ and $c\in K$, let $f:Q_0\to Q_1$ be defined by \eqref{Eq:f2}. It is not difficult to see that $f$ is a bijection. We claim that $f$ is a homomorphism. Indeed, $\bar A$ is additive, we have
\begin{align*}
    (a,u)f\cdot (b,v)f &= (a\bar A,c\cdot a\bar A + uA)(b\bar A, c\cdot b\bar A + vA)\\
        &= (a\bar A + b\bar A, (c\cdot a\bar A + uA)I_{b\bar A} + (c\cdot b\bar A + vA)J_{a\bar A})
 \end{align*}
and
\begin{displaymath}
    ((a,u)(b,v))f = (a+b,uI_b+vJ_a)f = ((a+b)\bar A, c\cdot (a+b)\bar A + (uI_b+vJ_a)A),
\end{displaymath}
so it remains to show $AI_{b\bar A} = I_bA$ and $AJ_{a\bar A}=J_aA$ for every $a$, $b\in W_0$. This follows from $A^{-1}M_aA = M_{a\bar A}$, and we conclude that $\Psi$ is well-defined.

Conversely, let $f:Q_0\to Q_1$ be an isomorphism. Corollary \ref{Cr:Associator} gives $\Asc(Q_0) = 0\times K = \Asc(Q_1)$, and so $(0\times K)f = 0\times K$. Hence there is a bijection $A:K\to K$ such that $(0,u)f=(0,uA)$ for every $u\in K$. Then $(0,uA+vA) = (0,uA)(0,vA) = (0,u)f(0,v)f = ((0,u)(0,v))f = (0,u+v)f=(0,(u+v)A)$ shows that $A$ is additive.

Let $B:W_0\to W_1$, $C:W_0\to K$ be such that $(a,0)f = (aB,aC)$ for every $a\in W_0$. Note that $(0,0)f=(0,0)$ implies $0B=0=0C$. Because $(a,u) = (a,0)(0,uJ_a^{-1})$, we must have
\begin{equation}\label{Eq:f}
    (a,u)f =  (a,0)f\cdot (0,uJ_a^{-1})f = (aB,aC)(0,uJ_a^{-1}A) = (aB, aC+uJ_a^{-1}AJ_{aB}).
\end{equation}
This proves that $B$ is onto $W_1$. Since
\begin{align*}
    ((a+b)B,(a+b)C) &= (a+b,0)f = ((a,0)(b,0))f = (a,0)f\cdot (b,0)f\\
        &= (aB,aC)(bB,bC) = (aB+bB,aCI_{bB} + bCJ_{aB}),
\end{align*}
$B$ is additive. To show that $B$ is one-to-one, suppose that $aB=bB$. Then $(a-b)B=0$ by additivity, and $a=b$ follows from the fact that $(0,K)f=(0,K)$.

We also deduce from the above equality that
\begin{equation}\label{Eq:AddC}
    (a+b)C = aC+aC\cdot bB + bC - bC\cdot aB.
\end{equation}
Using \eqref{Eq:AddC} and $(a+b)C=(b+a)C$, we obtain $2(aC\cdot bB) = 2(bC\cdot aB)$. If $\chr{k}\ne 2$, we deduce
\begin{equation}\label{Eq:CBBC}
    aC\cdot bB = bC\cdot aB.
\end{equation}
If $\chr{k}=2$, we can use \eqref{Eq:AddC} repeatedly to get
\begin{align*}
    bC &= ((a+b)+a)C = (a+b)C + aC + (a+b)C\cdot aB + aC\cdot (a+b)B\\
        &= (aC+bC+aC\cdot bB+bC\cdot aB) + aC + (aC+bC+aC\cdot bB+bC\cdot aB)\cdot aB\\
        &\quad + aC\cdot aB + aC\cdot bB\\
        &= bC + aC\cdot bB\cdot aB + bC\cdot aB\cdot aB.
\end{align*}
Hence $aC\cdot bB\cdot aB = bC\cdot aB\cdot aB$. When $a\ne 0$, we can cancel $aB\ne 0$ and deduce \eqref{Eq:CBBC}. When $a=0$, \eqref{Eq:CBBC} holds thanks to $0B=0=0C$.

Therefore, in either characteristic, we can fix an arbitrary $0\ne b\in W_0$ and obtain from \eqref{Eq:CBBC} the equality $aC = ((bB)^{-1}\cdot bC)\cdot aB$ for every $a\in W_0$. Hence $aC = c\cdot aB$ for some (unique) $c\in K$.

We proceed to show that
\begin{equation}\label{Eq:Conj}
    A^{-1}M_aA = M_{aB}
\end{equation}
for every $a\in W_0$. By \eqref{Eq:f},
\begin{align*}
    (a,u)f\cdot (b,v)f &= (aB,aC+uJ_a^{-1}AJ_{aB})(bB, bC + vJ_b^{-1}AJ_{bB})\\
     &= (aB+bB, (aC+uJ_a^{-1}AJ_{aB})I_{bB} + (bC+vJ_b^{-1}AJ_{bB})J_{aB})
\end{align*}
is equal to
\begin{displaymath}
    ((a,u)(b,v))f = (a+b,uI_b+vJ_a)f = ((a+b)B, (a+b)C + (uI_b+vJ_a)J_{a+b}^{-1}AJ_{(a+b)B}).
\end{displaymath}
Thus
\begin{displaymath}
    (a+b)C + (uI_b+vJ_a)J_{a+b}^{-1}AJ_{(a+b)B} = (aC + uJ_a^{-1}AJ_{aB})I_{bB} + (bC + vJ_b^{-1}AJ_{bB})J_{aB}.
\end{displaymath}
Since $(a+b)C = aCI_{bB}+bCJ_{aB}$ by \eqref{Eq:AddC}, the last equality simplifies to
\begin{displaymath}
    (uI_b+vJ_a)J_{a+b}^{-1}AJ_{(a+b)B} = uJ_a^{-1}AJ_{aB}I_{bB} + vJ_b^{-1}AJ_{bB}J_{aB}.
\end{displaymath}
With $v=0$ we obtain the equality of maps $K\to K$
\begin{equation}\label{Eq:AuxMps1}
    I_bJ_{a+b}^{-1}AJ_{(a+b)B} = J_a^{-1}AJ_{aB}I_{bB}.
\end{equation}
Similarly, with $u=0$ we deduce another equality of maps $K\to K$, namely
\begin{equation}\label{Eq:AuxMps2}
    J_aJ_{a+b}^{-1}AJ_{(a+b)B} = J_b^{-1}AJ_{bB}J_{aB}.
\end{equation}
Using both \eqref{Eq:AuxMps1} and \eqref{Eq:AuxMps2}, we see that
\begin{displaymath}
    I_b^{-1}J_a^{-1}AJ_{aB}I_{bB} = J_{a+b}^{-1}AJ_{(a+b)B} = J_a^{-1}J_b^{-1}AJ_{bB}J_{aB},
\end{displaymath}
and upon commuting certain maps and canceling we get $I_b^{-1}AI_{bB} = J_b^{-1}AJ_{bB}$, and therefore also $J_bAI_{bB} = I_bAJ_{bB}$. Upon expanding and canceling like terms, we get $2M_bA=2AM_{bB}$. If $\chr{k}\ne 2$, we deduce $M_bA=AM_{bB}$ and \eqref{Eq:Conj}. Suppose that $\chr{k}=2$. Then $\eqref{Eq:AuxMps1}$ with $a=b$ yields $I_bA = I_b^{-1}AI_{bB}I_{bB}$, so $I_b^2A = AI_{bB}^2$. Since $M_b^2 = M_{b^2}$ and $I_b^2=I_{b^2}$, we get $I_{b^2}A = AI_{(bB)^2}$, $M_{b^2}A = AM_{(bB)^2}$ and $A^{-1}M_{b^2} A = M_{(bB)^2}$. Since $K$ is perfect (this is the only time we use this assumption), the last equality shows that every $A^{-1}M_dA$ is of the form $M_e$, so, in particular, $A^{-1}M_bA = M_e$ for some $e$. Then $M_e^2 = (A^{-1}M_bA)^2 = A^{-1}M_b^2A = M_{bB}^2$, and evaluating this equality at $1$ yields $e^2=(bB)^2$ and $e=bB$. We have again established \eqref{Eq:Conj}.

Since $A:K\to K$ is an additive bijection, \eqref{Eq:Conj} holds and $\im(B)=W_1$, it follows that $A\in S(W_0,W_1)$ and $B=\bar A:W_0\to W_1$. We therefore have $(a,0)f = (a\bar A,c\cdot a\bar A)$, and $\Phi$ is well-defined by $(A,c)=f\Phi$.

It remains to show that $\Phi$ and $\Psi$ are mutual inverses. If $f\in\mathrm{Iso}(Q_0,Q_1)$ and $f\Phi = (A,c)$, then \eqref{Eq:Conj} yields $J_a^{-1}AJ_{aB} = A$. This means that \eqref{Eq:f} can be rewritten as \eqref{Eq:f2}, and thus $f\Phi\Psi = f$. Conversely, suppose that $(A,c)\in S(W_0,W_1)\times K$ and let $f=(A,c)\Psi$ and $(D,d) = f\Phi = (A,c)\Psi\Phi$. Then $(0,u)f = (0,uA)$ by \eqref{Eq:f2} and $(0,u)f = (0,uD)$ by definition of $\Phi$, so $A=D$. Finally, $(a,0)f = (a\bar A, c\cdot a\bar A)$ by \eqref{Eq:f2} and $(a,0)f = (a\bar D, d\cdot a\bar D) = (a\bar A, d\cdot a\bar A)$ by definition of $\Psi$, so $c=d$.
\end{proof}

\subsection{Isomorphisms and automorphisms in the tame case}

For the rest of this section suppose that the triple $k$, $K$, $W_i$ is tame, that is, $k$ is a prime field, $\langle W_i\rangle_k = K$, and $K$ is perfect if $\chr{k}=2$. In particular, $W_i\ne 0$. Let $\GL_k(K)$ be the group of all $k$-linear transformations of $K$, and let $\Aut(K)$ be the group of all field automorphisms of $K$.

Since $k$ is prime, any additive bijection $K\to K$ is $k$-linear, and so $S(W_0,W_1)=\{A\in\GL_k(K):A^{-1}M_{W_0}A = M_{W_1}\}$. We have shown that $A\in S(W_0,W_1)$ gives rise to an additive bijection $\bar A:W_0\to W_1$. This map extends uniquely into a field automorphism $\bar A$ of $K$ such that $A^{-1}M_aA = M_{a\bar A}$ for every $a\in K$. To see this, first note that $A\in\GL_k(K)$ implies $A^{-1}M_{ab}A = A^{-1}M_aM_bA = A^{-1}M_aAA^{-1}M_bA$, $A^{-1}M_{a+b}A = A^{-1}(M_a+M_b)A = A^{-1}M_aA + A^{-1}M_bA$ and $A^{-1}M_\lambda A = M_\lambda$ for every $a$, $b\in K$ and $\lambda\in k$. If $\bar A$ is already defined on $a$, $b$, let $(a+b)\bar A = a\bar A + b\bar A$, $(ab)\bar A = a\bar A\cdot b\bar A$, and $(\lambda a)\bar A = \lambda\cdot a\bar A$, where $\lambda\in k$. This procedure defines $\bar A$ well. For instance, if $ab=c+d$, we have $a\bar A\cdot b\bar A = 1M_{a\bar A\cdot b\bar A} = 1M_{a\bar A}M_{b\bar A} = 1A^{-1}M_aAA^{-1}M_bA = 1A^{-1}M_{ab}A = 1A^{-1}M_{c+d}A = 1(A^{-1}M_cA + A^{-1}M_dA) = 1(M_{c\bar A} + M_{d\bar A}) = c\bar A + d\bar A$, and so on.

Here is a solution to the isomorphism problem for a fixed extension $k<K$:

\begin{corollary}\label{Cr:Iso}
For $i\in\{0,1\}$, let $k$, $K$, $W_i$ be a tame triple and $Q_i = Q_{k<K}(W_i)$. Then $Q_0$ is isomorphic to $Q_1$ if and only if there is $\varphi\in\Aut(K)$ such that $W_0\varphi = W_1$.
\end{corollary}
\begin{proof}
Suppose that $f:Q_0\to Q_1$ is an isomorphism. By Proposition \ref{Pr:Correspondence}, $f$ induces a map $A\in S(W_0,W_1)$, which gives rise to $\bar A:W_0\to W_1$, which extends into $\bar A\in\Aut(K)$ such that $W_0\bar A = W_1$.

Conversely, suppose that $\varphi\in\Aut(K)$ satisfies $W_0\varphi = W_1$. Then for every $a\in W_0$ and $b\in K$ we have $b\varphi^{-1}M_a\varphi = ((b\varphi^{-1})\cdot a)\varphi = b\varphi^{-1}\varphi\cdot a\varphi = b\cdot a\varphi = bM_{a\varphi}$, so $\varphi\in S(W_0,W_1)$. The set $S(W_0,W_1)\times K$ is therefore nonempty, and we are done by Proposition \ref{Pr:Correspondence}.
\end{proof}

We proceed to describe the automorphism groups of tame loops $Q_{k<K}(W)$. Let $S(W) = S(W,W) = \{A\in\GL_k(K):A^{-1}M_WA=M_W\}$.

\begin{lemma}\label{Lm:Action}
Suppose that $k$, $K$, $W$ is a tame triple. Then the mapping $S(W)\to\Aut(K)$, $A\mapsto \bar A$ is a homomorphism with kernel $N(W) = M_{K^*}$ and image $I(W) = \{C\in\Aut(K):WC=W\}$. Moreover, $S(W) = I(W)N(W)$ is isomorphic to the semidirect product $I(W)\ltimes K^*$ with multiplication $(A,c)(B,d) = (A,c\bar B\cdot d)$.
\end{lemma}
\begin{proof}
With $A$, $B\in S(W)$ and $a\in K$ we have $M_{a\overline{AB}} = (AB)^{-1}M_a(AB) = B^{-1}A^{-1}M_a AB = B^{-1}M_{a\bar A}B = M_{a\bar A\bar B}$, so $\overline{AB} = \bar A\bar B$. The kernel of this homomorphism is equal to $N(W) = \{A\in S(W):M_aA = AM_a$ for every $a\in K\}$. If $A\in N(W)$, we can apply the defining equality to $1$ and deduce $aA = (1A)a$, so $A=M_{1A}\in M_{K^*}$. Conversely, if $M_b\in M_{K^*}$ then obviously $M_b\in N(W)$.

For the image, note that $\bar A$ satisfies $W\bar A = W$. We have seen above that $\bar A\in \Aut(K)$. Conversely, if $C\in\Aut(K)$ satisfies $WC=W$ then $C\in S(W)$, and $C^{-1}M_aC = M_{aC}$ for every $a\in K$ because $C$ is multiplicative. Thus $C=\bar C\in I(W)$.

Since $I(W)$, $N(W)$ are subsets of $S(W)$, we have $I(W)N(W)\subseteq S(W)$. To show that $S(W)\subseteq I(W)N(W)$, let $A\in S(W)$ and consider $D=(\bar A)^{-1}A\in S(W)$. Then $D^{-1}M_{a\bar A}D = A^{-1}\bar AM_{a\bar A}(\bar A)^{-1}A = A^{-1}M_a A = M_{a\bar A}$ shows that $D\in N(W)$. Then $A = \bar AD$ is the desired decomposition.

Let $A$, $B\in S(W) = I(W)N(W) = I(W)M_{K^*}$, where $A=\bar A M_c$, $B = \bar B M_d$ for some $c$, $d\in K^*$. Then $AB = \bar A M_c \bar B M_d = \bar A\bar B M_{c\bar B}M_d = \overline{AB} M_{c\bar B\cdot d}$.
\end{proof}

\begin{theorem}\label{Th:Aut}
Let $Q=Q_{k<K}(W)$, where $k$ is a prime field, $k<K$ is a field extension, $W$ is a $k$-subspace of $K$ such that $k1\cap W=0$, and $\langle W\rangle_k=K$. If $\chr{k}=2$, suppose also that $K$ is a perfect field. Then the group $\Aut(Q)$ is isomorphic to the semidirect product $S(W)\ltimes K$ with multiplication $(A,c)(B,d) = (AB, cB+d)$.
\end{theorem}
\begin{proof}
By Proposition \ref{Pr:Correspondence}, there is a one-to-one correspondence between the sets $\Aut(Q)$ and $S(W)\times K$. Suppose that $f\Phi = (A,c)$, $g\Phi = (B,d)$, so that $(a,u)f = (a\bar A, c\cdot a\bar A + uA)$ and $(a,u)g = (a\bar B, d\cdot a\bar B + uB)$ for every $(a,u)\in W\times K$. Then
\begin{displaymath}
    (a,u)fg = (a\bar A, c\cdot a\bar A + uA)g = (a\bar A\bar B,d\cdot a\bar A\bar B + (c\cdot a\bar A+uA)B).
\end{displaymath}
We want to prove that $(fg)\Phi = (AB, cB+d)$, which is equivalent to proving
\begin{displaymath}
    (a,u)fg = (a\overline{AB}, (cB+d)\cdot a\overline{AB} + uAB).
\end{displaymath}
Keeping $\bar A\bar B = \overline{AB}$ of Lemma \ref{Lm:Action} in mind, it remains to show that $(c\cdot a\bar A)B = cB\cdot a\bar A\bar B$, but this follows from $B^{-1}M_{a\bar A}B = M_{a\bar A\bar B}$.
\end{proof}

A finer structure of $\Aut(Q_{k<K}(W))$ is obtained by combining Theorem \ref{Th:Aut} with Lemma \ref{Lm:Action}.

\section{Automorphic loops of order $p^3$}\label{Sc:p3}

The following facts are known about automorphic loops of odd order and prime power order.

Automorphic loops of odd order are solvable \cite[Theorem 6.6]{KiKuPhVo}. Every automorphic loop of prime order $p$ is a group \cite[Corollary 4.12]{KiKuPhVo}. More generally, every automorphic loop of order $p^2$ is a group, by \cite{Cs} or \cite[Theorem 6.1]{KiKuPhVo}. For every prime $p$ there are examples of automorphic loops of order $p^3$ that are not centrally nilpotent \cite{KiKuPhVo}, and hence certainly not groups.

There is a \emph{commutative} automorphic loop of order $2^3$ that is not centrally nilpotent \cite{JeKiVoConst}. By \cite[Theorem 1.1]{JeKiVoNilp}, every commutative automorphic loop of odd order $p^k$ is centrally nilpotent. For any prime $p$ there are precisely $7$ commutative automorphic loops of order $p^3$ up to isomorphism \cite[Theorem 6.4]{BaGrVo}.

We will use a special case of Corollary \ref{Cr:Iso} to construct a class of pairwise non-isomorphic automorphic loops of odd order $p^3$, for $p$ odd.

Suppose that $p$ is odd. The field $\F_{p^2}$ can be represented as $\{x+y\sqrt{d}:x$, $y\in \F_p\}$, where $d\in\F_p$ is not a square. Let $\F_p=k<K=\F_{p^2}$, and let
\begin{displaymath}
    W_0 = k\sqrt d\text{ and } W_a= k(1+a\sqrt{d})\text{ for $0\ne a\in \F_p$.}
\end{displaymath}
We see that every $W_a$ is a $1$-dimensional $k$-subspace of $K$ such that $k1\cap W_a=0$. Conversely, if $W$ is a $1$-dimensional $k$-subspace of $K$ such that $k1\cap W=0$, there is $a+b\sqrt{d}$ in $W$ with $a$, $b\in k$, $b\ne 0$. If $a=0$ then $W=W_0$. Otherwise $a^{-1}(a+b\sqrt{d}) = 1+a^{-1}b\sqrt{d}\in W$, and $W=W_{a^{-1}b}$. Hence there is a one-to-one correspondence between the elements of $k$ and $1$-dimensional $k$-subspaces $W$ of $K$ satisfying $k1\cap W=0$, given by $a\mapsto W_a$.

\begin{theorem}\label{Th:p3}
Let $p$ be a prime and $\F_p=k<K=\F_{p^2}$.
\begin{enumerate}
\item[(i)] Suppose that $p$ is odd. If $a$, $b\in k$, then the automorphic loops $Q_{k<K}(W_a)$, $Q_{k<K}(W_b)$ of order $p^3$ are isomorphic if and only if $a=\pm b$. In particular, there are $(p+1)/2$ pairwise non-isomorphic automorphic loops of order $p^3$ of the form $Q_{k<K}(W)$, where we can take $W\in \{W_a:0\le a\le (p-1)/2\}$.
\item[(ii)] Suppose that $p=2$. Then there is a unique automorphic loop of order $p^3$ of the form $Q_{k<K}(W)$ up to isomorphism.
\end{enumerate}
\end{theorem}
\begin{proof}
(i) By Theorem \ref{Th:A}, the loops $Q_a=Q_{k<K}(W_a)$ and $Q_b = Q_{k<K}(W_b)$ are automorphic loops of order $p^3$. By Corollary \ref{Cr:Iso}, the loops $Q_a$, $Q_b$ are isomorphic if and only if there is an automorphism $\varphi$ of $K$ such that $W_a\varphi = W_b$. Let $\sigma$ be the unique nontrivial automorphism of $K$, given by $(a+b\sqrt{d})\sigma = a-b\sqrt{d}$. Then $W_a\sigma = W_{-a}$ for every $a\in k$. Therefore $Q_a$ is isomorphic to $Q_b$ if and only if $a=\pm b$. The rest follows.

Part (ii) is similar, and follows from Corollary \ref{Cr:Iso} by a direct inspection of subspaces and automorphisms of $\F_4$.
\end{proof}

We will now show how to obtain the loops of Construction \ref{Cn:JeKiVo} as a special case of Construction \ref{Cn:Extension}.

\begin{lemma}\label{Lm:Field}
Let $k$ be a field and $A\in M_2(k)\setminus kI$. Then $kI+kA$ is an anisotropic plane if and only if $kI+kA$ is a field with respect to the operations induced from $M_2(k)$.
\end{lemma}
\begin{proof}
Certainly $kI+kA$ is an abelian group. It is well known and easy to verify directly that every $A\in M_2(k)$ satisfies the characteristic equation
\begin{displaymath}
    A^2 = \tr(A)A -\det(A)I.
\end{displaymath}
This implies that $kI+kA$ is closed under multiplication, and it is therefore a subring of $M_2(k)$.

If $kI+kA$ is a field then every nonzero element $B\in kI+kA$ has an inverse in $kI+kA$, so $B$ is an invertible matrix and $kI+kA$ is an anisotropic plane. Conversely,
suppose that $kI+kA$ is an anisotropic plane, so that every nonzero element $B\in kI+kA$ is an invertible matrix. The characteristic equation for $B$ then implies that $B^{-1}=(\det(B)^{-1})(\tr(B)I-B)$, certainly an element of $kI+kA$, so $kI+kA$ is a field.
\end{proof}

\begin{proposition}\label{Pr:Same}
Let $k$ be a field. Let
\begin{align*}
    S &=\{Q_{k<K}(W): k<K \text{ is a quadratic field extension },\, \dim_k(W)=1,\, k1\cap W=0\},\\
    T &=\{Q_k(A):A\in M_2(k),\, kI+kA \text{ is an anisotropic plane}\}.
\end{align*}
Then, up to isomorphism, the loops of $S$ are precisely the loops of $T$.
\end{proposition}
\begin{proof}
Let $Q_{k<K}(W)\in S$. Then there is $\theta\in K$ such that $W=k\theta$, $K=k(\theta)$, and $\theta^2 = e+f\theta$ for some $e$, $f\in k$. The multiplication in $K$ is determined by $(a+b\theta)(c+d\theta) = (ac+bd\theta^2) + (ad+bc)\theta$ and $\theta^2=e+f\theta$. With respect to the basis $\{1,\theta\}$ of $K$ over $k$, the multiplication by $\theta$ is given by the matrix $A=M_\theta = \binom{0\ 1}{e\ f}$. The multiplication on $kI+kA$ is then determined by $(aI+bA)(cI+dA) = (acI+bdA^2)+(ad+bc)A$ and $A^2 = -\det(A)I+\tr(A)A = eI + fA$, so $kI+kA$ is a field isomorphic to $K$. By Lemma \ref{Lm:Field}, $kI+kA$ is an anisotropic plane, and the loop $Q_k(A)$ is defined.

The multiplication in $Q_{k<K}(W)$ on $W\times V = k\theta \times (k1 + k\theta)$ is given by $(a\theta,u)(b\theta,v) = (a\theta+b\theta, u(1+b\theta)+v(1-a\theta))$, while the multiplication in $Q_k(A)=Q_k(M_\theta)$ on $k\times (k\times k)$ is given by $(a,u)(b,v) = (a+b, u(1+b\theta)+v(1-a\theta))$. This shows that $Q_{k,K}(W)$ is isomorphic to $Q_k(A)$, and $S\subseteq T$.

Conversely, if $Q_k(A)\in T$ then the anisotropic plane $K=kI+kA$ is a field by Lemma \ref{Lm:Field}, clearly a quadratic extension of $k$. Moreover, $W=kA$ is a $1$-dimensional $k$-subspace of $K$ such that $k1\cap W=0$, so $Q_{k<K}(W)\in S$. We can again show that $Q_{k<K}(W)$ is isomorphic to $Q_k(A)$.
\end{proof}

Conjecture 6.5 of \cite{JeKiVoNilp} stated that there is precisely one isomorphism type of loops $Q_{\F_2}(A)$, two isomorphism types of loops $Q_{\F_3}(A)$, and three isomorphism types of loops $Q_{\F_p}(A)$ for $p\ge 5$. The conjecture was verified computationally in \cite{JeKiVoNilp} for $p\le 5$, using the GAP package \texttt{LOOPS} \cite{LOOPS}. Since $\F_{p^2}$ is the unique quadratic extension of $\F_p$, Theorem \ref{Th:p3} and Proposition \ref{Pr:Same} now imply that the conjecture is actually false for every $p>5$. (But note that $(p+1)/2$ gives the calculated answer for $p=3$ and $p=5$, and the case $p=2$ is also in agreement.)

The full classification of automorphic loops of order $p^3$ remains open.

\section{Infinite examples}\label{Sc:Example}

We conclude the paper by constructing an infinite $2$-generated abelian-by-cyclic automorphic loop of exponent $p$ for every prime $p$.

\begin{lemma}
Let $p$ be an odd prime, $k=\F_p$, $K=\F_p((t))$ the field of formal Laurent series over $\F_p$, $W=\F_pt$, and $Q=Q_{k<K}(W)$ the automorphic loop from Construction \ref{Cn:Extension} defined by \eqref{Eq:Mult} on $W\times K = \F_pt\times \F_p((t))$. Let $L=\langle (t,0),(0,1)\rangle$ be the subloop of $Q$ generated by $(t,0)$ and $(0,1)$. Then $L=W\times U$, where $U$ is the localization of $\F_p[t]$ with respect to $\{1+a:a\in W\}$. Moreover, $L$ is an infinite nonassociative $2$-generated abelian-by-cyclic automorphic loop of exponent $p$.
\end{lemma}
\begin{proof}
First we observe that $W\times U$ is a subloop of $Q$. Indeed, $W\times U$ is clearly closed under multiplication. Since $(1\pm a)^{-1}\in U$ for every $a\in W$ by definition, the formulas \eqref{Eq:LeftDivision}, \eqref{Eq:RightDivision} show that $W\times U$ is closed under left and right divisions, respectively. To prove that $L=W\times U$, it therefore suffices to show that $W\times U\subseteq L$.

We claim that $0\times \F_p[t]\subseteq L$, or, equivalently, that $(0,t^n)\in L$ for every $n\ge 0$. First note that for any integer $m$ we have
\begin{equation}\label{Eq:AuxPower}
    (0,t^m)(t,0)\cdot (t,0)^{-1}(0,t^m) = (t,t^m(1+t))(-t,t^m(1+t)) = (0,2(t^m-t^{m+2})).
\end{equation}
We have $(0,t^0)=(0,1)\in L$ by definition. The identity \eqref{Eq:AuxPower} with $m=0$ then yields $(0,2(1-t^2))\in L$, so $(0,t^2)\in L$. Since also
\begin{displaymath}
    (-t,0)\cdot (0,1)(t,0) = (-t,0)(t,1+t) = (0,1+2t+t^2)
\end{displaymath}
belongs to $L$, we conclude that $(0,t)\in L$. The identity \eqref{Eq:AuxPower} can then be used inductively to show that $(0,t^n)\in L$ for every $n\ge 0$.

We now establish $0\times U\subseteq L$ by proving that $(0,(1+a)^n)\in L$ for every $n\in\mathbb Z$ and every $a\in W=\F_p$. We have already seen this for $n\ge 0$. The identity
\begin{displaymath}
    ((a,0)\backslash (0,(1-a)^m))/(-a,0) = (-a, (1-a)^{m-1})/(-a,0) = (0,(1-a)^{m-2})
\end{displaymath}
then proves the claim by descending induction on $m$, starting with $m=1$.

Given $(a,0)\in W\times 0\subseteq L$ and $(0,u)\in 0\times U\subseteq L$, we note that $(0,u(a(1-a)^{-1}))\in L$, and thus
\begin{displaymath}
    (a,0)(0,u)\cdot(0,u(a(1-a)^{-1})) = (a,u(1-a))(0,u(a(1-a)^{-1})) = (a,u)
\end{displaymath}
is also in $L$, concluding the proof that $W\times U\subseteq L$.

The loop $L$ is certainly infinite and $2$-generated, and it is automorphic by Theorem \ref{Th:A}. The homomorphism $W\times U\to \F_p$, $(it,u)\mapsto i$ has the abelian group $(U,+)$ as its kernel and the cyclic group $(\F_p,+)$ as its image, so $L$ is abelian-by-cyclic. An easy induction yields $(a,u)^m = (ma,mu)$ for every $(a,u)\in Q$ and $m\ge 0$, proving that $L$ has exponent $p$. Finally, $(t,0)(t,0)\cdot (0,1) = (2t,1-2t)\ne (2t,1-2t+t^2) = (t,0)\cdot (t,0)(0,1)$ shows that $L$ is nonassociative.
\end{proof}

\begin{lemma}
Let $k=\F_2$, $K=\F_2((t))$ the field of formal Laurent series over $\F_2$, $W=\F_2t$, and $Q=Q_{k<K}(W)$ the automorphic loop from Construction \ref{Cn:Extension} defined by \eqref{Eq:Mult} on $W\times K = \F_2t\times \F_2((t))$. Let $L=\langle (t,0),(0,1)\rangle$ be the subloop of $Q$ generated by $(t,0)$ and $(0,1)$. Then $L= \{(it,f(1+t)^i):f\in U,i\in\{0,1\}\}$, where $U$ is the localization of $\F_2[t^2]$ with respect to $\{1+t^2\}$. Moreover, $L$ is an infinite nonassociative $2$-generated abelian-by-cyclic commutative automorphic loop of exponent $2$.
\end{lemma}
\begin{proof}
In our situation the multiplication formula \eqref{Eq:Mult} becomes
\begin{displaymath}
    (a,u)(b,v) = (a+b, u(1+b)+v(1+a)),
\end{displaymath}
so $Q$ is commutative and of exponent $2$. Note that \eqref{Eq:LeftDivision} becomes
\begin{displaymath}
    (a,u)\backslash (b,v) = (a+b, (v+u(1+a+b))(1+a)^{-1}).
\end{displaymath}
Let us first show that $S=\{(it,f(1+t)^i):f\in U,i\in\{0,1\}\} = (0\times U)\cup (t,0)(0\times U)$ is a subloop of $Q$. Indeed, $0\times U\subseteq S$ is a subloop, and with $f$, $g\in U$, we have
\begin{align*}
    (t,f(1+t))(t,g(1+t)) &= (0, f(1+t)^2+g(1+t)^2) = (0,(f+g)(1+t^2)),\\
    (0,f)\backslash (t,g(1+t)) &= (t, g(1+t)+f(1+t)) = (t, (g+f)(1+t)),\\
    (t,f(1+t))\backslash (0,g) &= (t, (g+f(1+t)^2)(1+t)^{-1}) = (t, (g(1+t^2)^{-1}+f)(1+t)),\\
    (t,f(1+t))\backslash (t,g(1+t)) &= (0,(g(1+t)+f(1+t))(1+t)^{-1}) = (0,g+f),
\end{align*}
always obtaining an element of $S$.

To prove that $S=L$, it suffices to show that $(0,t^{2m})$, $(0,t^{2m}(1+t^2)^{-1})\in L$ for every $m\ge 0$, since this implies $0\times U\subseteq L$ and thus $S=(0\times U)\cup (t,0)(0\times U)\subseteq L$.
We have $(0,1)\in L$ by definition, $(t,1+t) = (t,0)(0,1)\in L$, $(t,(1+t^2)^{-1}(1+t)) = (t,0)\backslash (0,1)\in L$, and $(0,1+(1+t^2)^{-1}) = (t,1+t)\backslash (t,(1+t^2)^{-1}(1+t))\in L$, so also $(0,(1+t^2)^{-1})\in L$. The inductive step follows upon observing the identity
\begin{displaymath}
    (t,0)\cdot(0,u)(t,0) =(t,0)(t,u(1+t)) = (0,u(1+t^2)).
\end{displaymath}

The loop $L$ is certainly infinite, $2$-generated, commutative, automorphic and of exponent $2$. It is abelian-by-cyclic because the map $L\to \F_2$, $(it,f(1+t)^i)\mapsto i$ is a homomorphism with the abelian group $(U,+)$ as its kernel and the cyclic group $(\F_2,+)$ as its image. Finally, $(t,0)(t,0)\cdot (0,1) = (0,1)\ne (0,1+t^2) = (t,0)\cdot (t,0)(0,1)$ shows that $L$ is nonassociative.
\end{proof}

\section*{Acknowledgment}

We thank the three anonymous referees for useful comments, particularly for pointing out a counting mistake in an earlier version of Theorem \ref{Th:p3}(ii).

\end{document}